\documentclass[]{amsart}
\usepackage{amsfonts}
\usepackage{amssymb}
\usepackage{amsmath}
\usepackage{amsthm}
\usepackage{amscd}
\usepackage{mathrsfs}
\usepackage[all]{xy}
\usepackage{graphicx, color}
\usepackage{url}

\allowdisplaybreaks[0]

\theoremstyle{theorem}
\newtheorem{theorem}{Theorem}[section]

\newtheorem{proposition}[theorem]{Proposition}
\theoremstyle{definition}
\newtheorem{definition}[theorem]{Definition}
\theoremstyle{remark}
\newtheorem{remark}[theorem]{Remark}

\newcommand{\Aut}[1]{\mathrm{Aut}(#1)}
\newcommand{\Inn}[1]{\mathrm{Inn}(#1)}
\newcommand{\Hom}[1]{\mathrm{Hom}(#1)}

\begin{document}

\title{Quasi-triviality of quandles for link-homotopy}
\author[Ayumu Inoue]{Ayumu Inoue${}^{\dagger}$}
\thanks{${}^{\dagger}$The author is partially supported by Grant-in-Aid for Research Activity Start-up, No.\ 23840014, Japan Society for the Promotion of Science.}
\address{Department of Mathematical and Computing Sciences, Tokyo Institute of Technology, Ookayama, Meguro--ku, Tokyo, 152--8552 Japan}
\email{ayumu.inoue@math.titech.ac.jp}


\begin{abstract}
We introduce the notion of quasi-triviality of quandles and define homology of quasi-trivial quandles.
Quandle cocycle invariants are invariant under link-homotopy if they are associated with 2-cocycles of quasi-trivial quandles.
We thus obtain a lot of numerical link-homotopy invariants.
\end{abstract}

\maketitle

\section{Introduction}
\label{sec:introduction}

Link-homotopy, introduced by Milnor \cite{Mil1954}, gives rise to an equivalence relation on oriented and ordered links in a 3-sphere.
More precisely, two links are said to be link-homotopic if they are related to each other by a finite sequence of ambient isotopies and self-crossing changes, keeping the orientation and ordering.
Here, a self-crossing change is a homotopy for a single component of a link, supported in a small ball whose intersection with the component consists of two segments, depicted in Figure \ref{fig:self-crossing_change}.
The classification problem of links up to link-homotopy is already solved by Habegger and Lin \cite{HL1990} completely.
They gave an algorithm which determines whether given links are link-homotopic or not.
On the other hand, a table consisting of representatives of all link-homotopy classes is still not known other than partial ones given by Milnor \cite{Mil1954, Mil1957} for links with 3 or fewer components and by Levine \cite{Lev1988} for links with 4 components.
We note that both of them utilized numerical invariants to obtain the tables.
\begin{figure}[htbp]
\begin{center}
\includegraphics[scale=0.2]{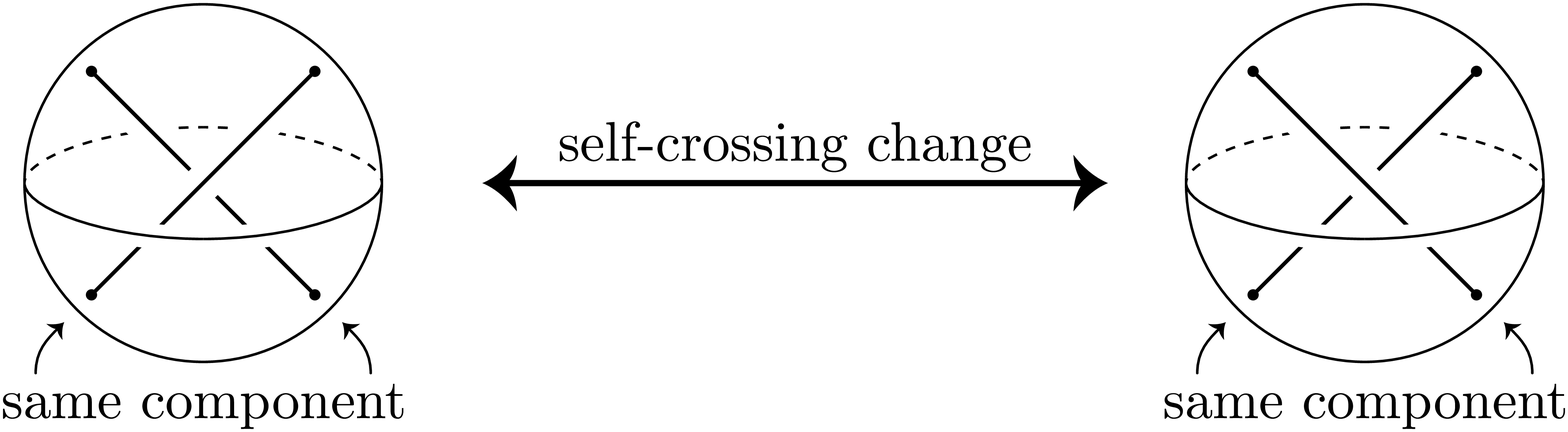}
\end{center}
\caption{}
\label{fig:self-crossing_change}
\end{figure}

A quandle, given by Joyce \cite{Joy1982}, is an algebraic system consisting of a set together with a binary operation whose definition is strongly motivated in knot theory.
Carter et al.\ \cite{CJKLS2003} introduced quandle homology and gave invariants of links up to ambient isotopy using 2-cocycles of quandles.
These invariants, called quandle cocycle invariants, are not invariant under link-homotopy in general.
However, in this paper, we introduce the notion of quasi-triviality of quandles and show that quandle cocycle invariants associated with 2-cocycles of quasi-trivial quandles are invariant under link-homotopy (Theorem \ref{thm:main}), modifying the definition of quandle homology slightly.
We thus have a lot of numerical link-homotopy invariants, which might enable us to obtain a table consisting of representatives of all link-homotopy classes.
At the end of this paper, as an application of Theorem \ref{thm:main}, we show a famous fact that the Borromean rings, which is a 3-component link, is not trivial up to link-homotopy.

Throughout this paper, every link is assumed to be oriented, ordered and in a 3-sphere.
We also assume that every link diagram is oriented and ordered.

\section{Preliminaries}
\label{sec:preliminaries}

We devote this section to reviewing a quandle cocycle invariant.
Recall that two links are ambient isotopic if and only if their diagrams are related to each other by a finite sequence of Reidemeister moves.

We first review the definition of a quandle.
A \emph{quandle} is a non-empty set $X$ equipped with a binary operation $\ast : X \times X \rightarrow X$ satisfying the following axioms:
\begin{itemize}
\item[(Q1)]
For each $x \in X$, $x \ast x = x$.
\item[(Q2)]
For each $x \in X$, a map $\ast \> x : X \rightarrow X$ ($\bullet \mapsto \bullet \ast x$) is bijective.
\item[(Q3)]
For each $x, y, z \in X$, $(x \ast y) \ast z = (x \ast z) \ast (y \ast z)$.
\end{itemize}
The notion of homomorphisms of quandles is appropriately defined.
The axioms (Q1), (Q2) and (Q3) of a quandle are closely related to the Reidemeister moves RI, RII and RIII respectively as follows.

An \emph{arc coloring} of a link diagram $D$ by a quandle $X$ is a map $\{ \textrm{arcs of $D$} \} \rightarrow X$ satisfying the condition depicted in Figure \ref{fig:condition_of_arc_coloring} at each crossing.
We call an element of a quandle assigned to an arc by an arc coloring a \emph{color} of the arc.
Suppose $D$ is a link diagram and $D^{\prime}$ a diagram obtained from $D$ by a Reidemeister move.
Then, for each arc coloring $\mathscr{A}$ of $D$, we have a unique arc coloring of $D^{\prime}$ whose restriction to arcs unrelated to the deformation coincides with the restriction of $\mathscr{A}$.
Indeed, the axioms (Q1), (Q2) and (Q3) of a quandle guarantee that we can perform RI, RII and RIII moves fixing colors of ends respectively (See Figure \ref{fig:correspondence_of_arc_colorings}).
Therefore, for a fixed quandle, the number of all arc colorings is invariant under ambient isotopy.
\begin{figure}[htbp]
\begin{center}
\includegraphics[scale=0.2]{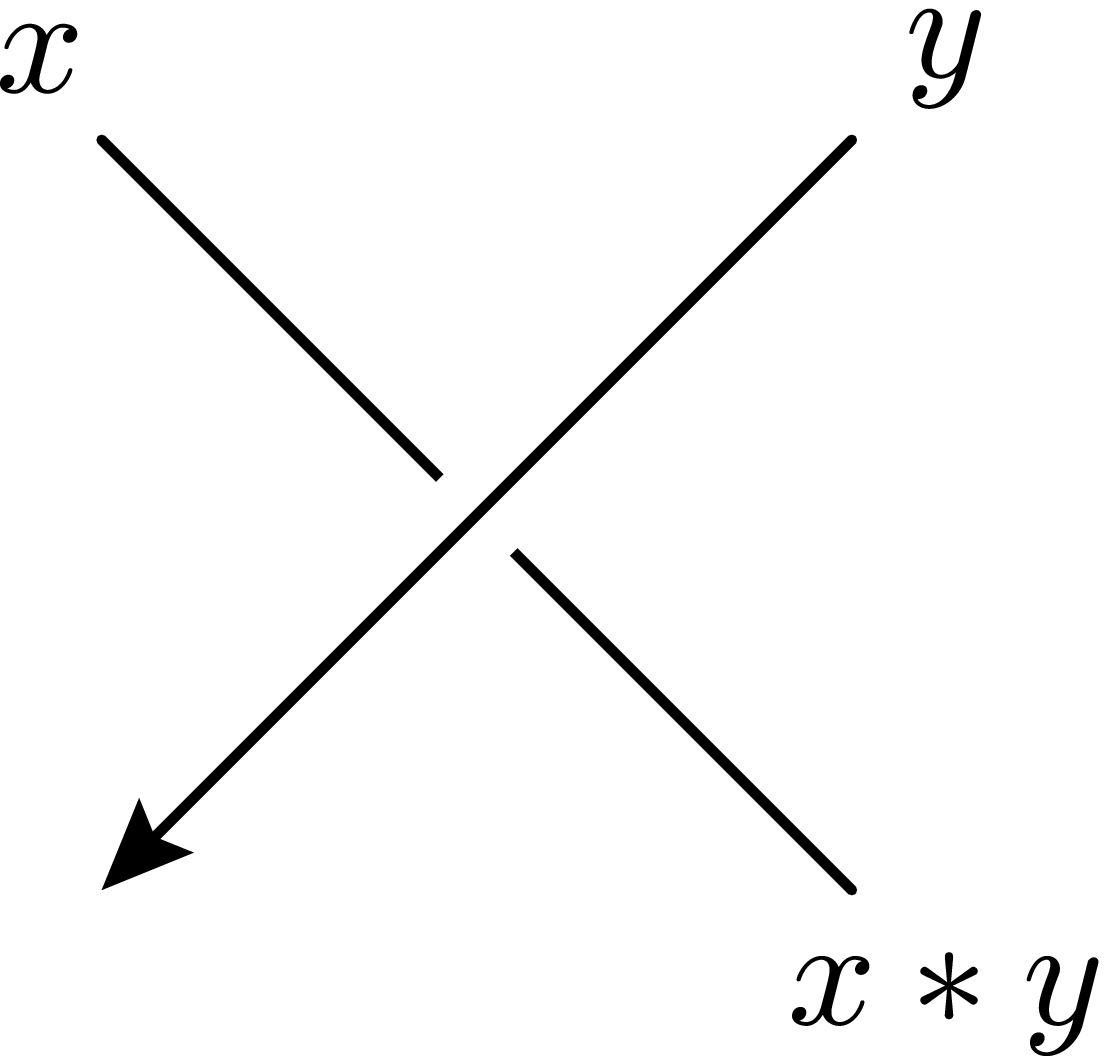}
\end{center}
\caption{}
\label{fig:condition_of_arc_coloring}
\end{figure}
\begin{figure}[htbp]
\begin{center}
\includegraphics[scale=0.2]{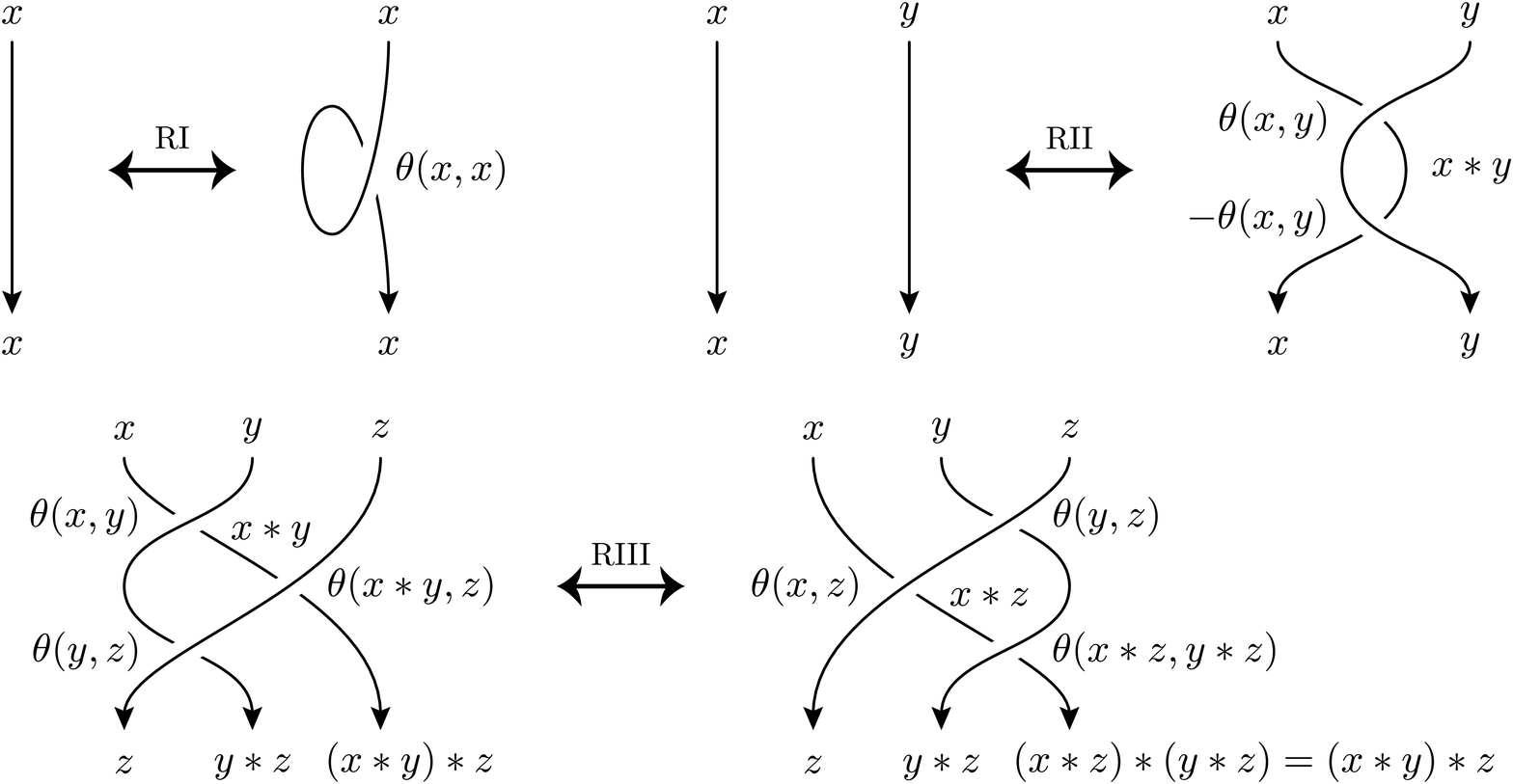}
\end{center}
\caption{}
\label{fig:correspondence_of_arc_colorings}
\end{figure}

We next review quandle homology.
Let $X$ be a quandle.
Consider the free abelian group $C^{R}_{n}(X)$ generated by all $n$-tuples $(x_{1}, x_{2}, \dots, x_{n}) \in X^{n}$ for each $n \geq 1$.
We let $C^{R}_{0}(X) = \mathbb{Z}$.
Define a map $\partial_{n} : C^{R}_{n} \rightarrow C^{R}_{n - 1}$ by
\begin{eqnarray*}
& & \partial_{n}(x_{1}, x_{2}, \dots, x_{n}) = \sum_{i = 2}^{n} (-1)^{i} \{ (x_{1}, \dots, x_{i-1}, x_{i+1}, \dots, x_{n}) \\
& & \hskip 15em - \, (x_{1} \ast x_{i}, \dots, x_{i-1} \ast x_{i}, x_{i+1}, \dots, x_{n}) \}
\end{eqnarray*}
for $n \geq 2$, and $\partial_{1} = 0$.
Then $\partial_{n-1} \circ \partial_{n} = 0$.
Thus $(C^{R}_{n}(X), \partial_{n})$ is a chain complex.
Let $C^{D}_{n}(X)$ be a subgroup of $C^{R}_{n}(X)$ generated by $n$-tuples $(x_{1}, x_{2}, \dots, x_{n}) \in X^{n}$ with $x_{i} = x_{i+1}$ for some $i$ if $n \geq 2$, and let $C^{D}_{n}(X) = 0$ otherwise.
It is routine to check that $\partial_{n}(C^{D}_{n}(X)) \subset C^{D}_{n-1}(X)$.
Therefore, putting $C^{Q}_{n}(X) = C^{R}_{n}(X) / C^{D}_{n}(X)$, we have a chain complex $(C^{Q}_{n}(X), \partial_{n})$.
Let $A$ be an abelian group.
The $n$-th \emph{quandle homology group} $H^{Q}_{n}(X; A)$ with coefficients in $A$ is the $n$-th homology group of the chain complex $(C^{Q}_{n}(X) \otimes A, \partial_{n} \otimes \textrm{id})$.
The $n$-th \emph{quandle cohomology group} $H_{Q}^{n}(X; A)$ with coefficients in $A$ is the $n$-th cohomology group of the cochain complex $(\Hom{C^{Q}_{n}(X), A}, \Hom{\partial_{n}, \textrm{id}})$.
We note that $\theta \in \Hom{C^{Q}_{2}(X), A}$ is a 2-cocycle if and only if $\theta$ satisfies the following conditions:
\begin{itemize}
\item[(C1)]
For each $x \in X$, $\theta(x, x) = 0$.
\item[(C2)]
For each $x, y, z \in X$, $\theta(\partial(x, y, z)) = 0$, \\
that is, $\theta(x, y) + \theta(x \ast y, z) = \theta(x, z) + \theta(x \ast z, y \ast z)$.
\end{itemize}

Associated with a 2-cocycle of a quandle, we can define an weight of an arc coloring as follows.
Suppose $X$ is a quandle and $\theta \in \Hom{C^{Q}_{2}(X), A}$ a 2-cocycle with coefficients in an abelian group $A$.
The $i$-th \emph{weight} of an arc coloring $\mathscr{A}$ of a link diagram $D$ by $X$ associated with $\theta$ is a value
\[
W(\mathscr{A}, \theta; i) = \sum \varepsilon \cdot \theta(x, y) \ \in A,
\]
where the sum runs over the crossings of $D$, each of which consists of under arcs belonging to the $i$-th component and an over arc, $\varepsilon$ is $1$ or $-1$ depending on whether the crossing is positive or negative respectively, and $x, y \in X$ denote colors of arcs around the crossing as depicted in Figure \ref{fig:condition_of_arc_coloring}.
We have the following theorem.

\begin{theorem}[Carter et al.\ \cite{CJKLS2003}]
\label{thm:CJKLS2003}
Let $X$ be a quandle and $\theta \in \Hom{C^{Q}_{2}(X), A}$ a 2-cocycle with coefficients in an abelian group $A$.
Then an weight of an arc coloring by $X$ associated with $\theta$ is invariant under Reidemeister moves.
\end{theorem}

\begin{proof}
Let $\mathscr{A}$ be an arc coloring of a link diagram by $X$.
Consider to perform a Reidemeister move for this colored diagram.
An RI move for a segment of the $i$-th component only adds or subtracts $\theta(x, x)$ to or from $W(\mathscr{A}, \theta; i)$ with some $x \in X$ (See the upper left of Figure \ref{fig:correspondence_of_arc_colorings}).
It does not change $W(\mathscr{A}, \theta; i)$ because $\theta$ satisfies the condition (C1).
An RII move, by which a segment of the $i$-th component passes under a some arc, only adds or subtracts $\theta(x, y) - \theta(x, y) = 0$ to or from $W(\mathscr{A}, \theta; i)$ with some $x, y \in X$ (See the upper right of Figure \ref{fig:correspondence_of_arc_colorings}).
An RIII move adds $\pm \theta(\partial(x, y, z))$ to $W(\mathscr{A}, \theta; i)$ with some $x, y, z \in X$ if the innermost arcs related to the deformation belong to the $i$-th component (See the bottom of Figure \ref{fig:correspondence_of_arc_colorings}).
Since $\theta$ satisfies the condition (C2), it does not change $W(\mathscr{A}, \theta; i)$.
\end{proof}

Suppose again that $X$ is a quandle and $\theta \in \Hom{C^{Q}_{2}(X), A}$ a 2-cocycle with coefficients in an abelian group $A$.
Theorem \ref{thm:CJKLS2003} says that, for each link $L$ and index $i$, the multiset $\Phi(L, \theta; i)$ consisting of $i$-th weights of all arc coloring of a diagram of $L$ by $X$ associated with $\theta$ does not depend on the choice of the diagram.
Thus $\Phi(L, \theta; i)$ is invariant under ambient isotopy.
The $i$-th \emph{quandle cocycle invariant} of $L$ associated with $\theta$ is this multiset $\Phi(L, \theta; i)$.

\section{Quasi-trivial quandle and link-homotopy}
\label{sec:quasi-trivial_quandle_and_link-homotopy}

Although the number of all arc colorings is invariant under ambient isotopy, it is not invariant under self-crossing changes in general.
Indeed, the trefoil knot is deformed into the unknot by a self-crossing change, but they are distinguished by the numbers of all arc colorings.
Thus a quandle cocycle invariant is not invariant under link-homotopy in general.
On the other hand, in this section, we show that a self-crossing change on a diagram also relates arc colorings of the original and deformed diagrams uniquely if we use a certain quandle, named a quasi-trivial quandle, for arc colorings.
Since two links are link-homotopic if and only if their diagrams are related to each other by a finite sequence of Reidemeister moves and self-crossing changes on diagrams, the number of all arc colorings by a quasi-trivial quandle is invariant under link-homotopy.
Furthermore, modifying the definition of quandle homology slightly, we will have a quandle cocycle invariant which is invariant under link-homotopy.

To define a quasi-trivial quandle, we first review the automorphism group and inner automorphism group of a quandle.
The \emph{automorphism group} of a quandle $X$ is a group $\Aut{X}$ consisting of all automorphisms of $X$ together with a product given by the composition of maps.
The axioms (Q2) and (Q3) of a quandle guarantee that, for each $x \in X$, the bijection $\ast \> x : X \rightarrow X$ is an automorphism of $X$.
The \emph{inner automorphism group} $\Inn{X}$ is the subgroup of $\Aut{X}$ generated by all automorphisms $\ast \> x : X \rightarrow X$.
We call an element of $\Inn{X}$ an \emph{inner automorphism} of $X$.

\begin{definition}
A quandle $X$ is said to be \emph{quasi-trivial} if $x \ast \varphi(x) = x$ for each $x \in X$ and $\varphi \in \Inn{X}$.
\end{definition}

As mentioned above, we have the following proposition.

\begin{proposition}
\label{prop:invariant_under_self-crossing_change_1}
Let $X$ be a quasi-trivial quandle.
Suppose $D$ is a link diagram and $D^{\prime}$ a diagram obtained from $D$ by a self-crossing change.
Then, for each arc coloring $\mathscr{A}$ of $D$ by $X$, we have a unique arc coloring of $D^{\prime}$ whose restriction to arcs unrelated to the deformation coincides with the restriction of $\mathscr{A}$.
\end{proposition}

\begin{proof}
Assume that the self-crossing change is performed at a crossing $c$.
Let $x, y, x \ast y \in X$ denote colors of arcs around $c$ by $\mathscr{A}$ as depicted in Figure \ref{fig:condition_of_arc_coloring}.
Since under arcs and an over arc around $c$ belong to the same component, we have
\[
y = (\ast \, z_{n})^{\varepsilon_{n}} \circ \dots \circ (\ast \, z_{2})^{\varepsilon_{2}} \circ (\ast \, z_{1})^{\varepsilon_{1}} (x)
\]
with some $n \in \mathbb{Z}_{\geq 0}$, $z_{i} \in X$ and $\varepsilon_{i} \in \{ \pm 1 \}$ (See the left-hand side of Figure \ref{fig:invariance_under_self-crossing_change}).
Let $\varphi \in \Inn{X}$ denote the inner automorphism $(\ast \, z_{n})^{\varepsilon_{n}} \circ \dots \circ (\ast \, z_{2})^{\varepsilon_{2}} \circ (\ast \, z_{1})^{\varepsilon_{1}}$.
By the assumption that $X$ is quasi-trivial, $x \ast y = x \ast \varphi(x) = x$.
Thus, remarking that $\varphi(x) \ast x = \varphi(x)$, we have a unique arc coloring $\mathscr{A}^{\prime}$ of $D^{\prime}$ whose restriction to arcs unrelated to the deformation coincides with the restriction of $\mathscr{A}$ (See the right-hand side of Figure \ref{fig:invariance_under_self-crossing_change}).
\begin{figure}
\begin{center}
\includegraphics[scale=0.18]{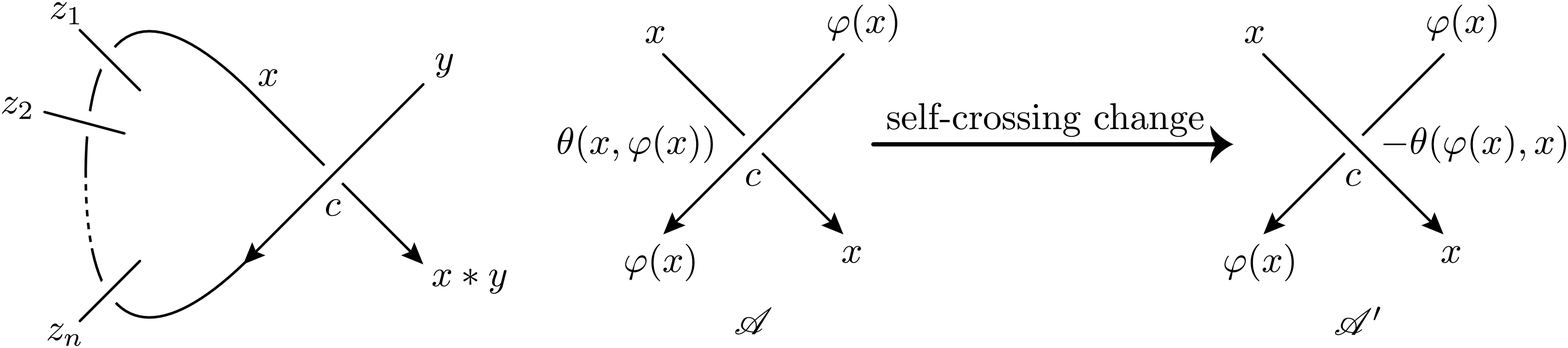}
\end{center}
\caption{}
\label{fig:invariance_under_self-crossing_change}
\end{figure}
\end{proof}

\begin{remark}
For a link $L$, we have the \emph{knot quandle} $Q(L)$, defined by Joyce \cite{Joy1982}, which is invariant under ambient isotopy (The same notion is given by Matveev \cite{Mat1982}).
An arc coloring of a diagram of $L$ by a quandle $X$ is just a diagrammatic presentation of a homomorphism $Q(L) \rightarrow X$.

We further have the \emph{reduced knot quandle} $RQ(L)$ of $L$, given by Hughes \cite{Hug2011}, which is the quasi-trivialization of $Q(L)$.
Suppose $X$ is a quasi-trivial quandle.
Then each homomorphism $Q(L) \rightarrow X$ factors through a homomorphism $RQ(L) \rightarrow X$.
Conversely, each homomorphism $RQ(L) \rightarrow X$ uniquely lifts to a homomorphism $Q(L) \rightarrow X$.
Thus an arc coloring of a diagram of $L$ by $X$ is just a diagrammatic presentation of a homomorphism $RQ(L) \rightarrow X$.
Hughes showed that $RQ(L)$ is invariant under link-homotopy.
It gives an alternative proof for Proposition \ref{prop:invariant_under_self-crossing_change_1}.
\end{remark}

Proposition \ref{prop:invariant_under_self-crossing_change_1} says that, for a fixed quasi-trivial quandle, the number of all arc colorings is invariant under link-homotopy.
However, weights of arc colorings are not invariant under self-crossing changes in general, even if we use quasi-trivial quandles for arc colorings.
We thus have to modify the definition of quandle homology slightly, as follows, so that weights are invariant under self-crossing changes.

Let $X$ be a quasi-trivial quandle.
Define the free abelian groups $C^{R}_{n}(X)$ and boundary maps $\partial_{n} : C^{R}_{n}(X) \rightarrow C^{R}_{n-1}(X)$ as in Section \ref{sec:preliminaries}.
Let $C^{D, qt}_{n}(X)$ be a subgroup of $C^{R}_{n}(X)$ generated by $n$-tuples $(x_{1}, x_{2}, \dots, x_{n}) \in X^{n}$ with $x_{i} = x_{i+1}$ for some $i$ and $n$-tuples $(x_{1}, \varphi(x_{1}), x_{3}, \dots, x_{n}) \in X^{n}$ for some $\varphi \in \Inn{X}$ if $n \geq 2$.
Let $C^{D, qt}_{n}(X) = 0$ if $n = 0, 1$.
It is easy to see that $\partial_{n}(C^{D, qt}_{n}(X)) \subset C^{D, qt}_{n-1}(X)$.
Thus, putting $C^{Q, qt}_{n}(X) = C^{R}_{n}(X) / C^{D, qt}_{n}(X)$, we have a chain complex $(C^{Q, qt}_{n}(X), \partial_{n})$.

Suppose that $A$ is an abelian group.
The $n$-th \emph{quasi-trivial quandle homology group} $H^{Q, qt}_{n}(X; A)$ with coefficients in $A$ is the $n$-th homology group of the chain complex $(C^{Q, qt}_{n}(X) \otimes A, \partial_{n} \otimes \textrm{id})$.
The $n$-th \emph{quasi-trivial quandle cohomology group} $H_{Q, qt}^{n}(X; A)$ with coefficients in $A$ is the $n$-th cohomology group of the cochain complex $(\Hom{C^{Q, qt}_{n}(X), A}, \Hom{\partial_{n}, \textrm{id}})$.
We note that $\theta \in \Hom{C^{Q, qt}_{2}(X), A}$ is a 2-cocycle if and only if $\theta$ satisfies the conditions (C1), (C2) in Section \ref{sec:preliminaries} and the following condition:
\begin{itemize}
\item[(C3)]
For each $x \in X$ and $\varphi \in \Inn{X}$, $\theta(x, \varphi(x)) = 0$.
\end{itemize}

\begin{proposition}
\label{prop:invariant_under_self-crossing_change_2}
Let $X$ be a quasi-trivial quandle and $\theta \in \Hom{C^{Q, qt}_{2}(X), A}$ a 2-cocycle with coefficients in an abelian group $A$.
Then an weight of an arc coloring by $X$ associated with $\theta$ is invariant under self-crossing changes on diagrams.
\end{proposition}

\begin{proof}
Let $\mathscr{A}$ be an arc coloring of a link diagram by $X$.
A self-crossing change at a crossing of the diagram subtracts $\pm (\theta(x, \varphi(x)) + \theta(\varphi(x), x))$ from $W(\mathscr{A}, \theta; i)$ with some $x \in X$ and $\varphi \in \Inn{X}$ if the crossing consists of arcs belonging to the $i$-th component (See the right-hand side of Figure \ref{fig:invariance_under_self-crossing_change}).
Since $\theta$ satisfies the condition (C3), it does not change $W(\mathscr{A}, \theta; i)$.
\end{proof}

By Theorem \ref{thm:CJKLS2003} and Proposition \ref{prop:invariant_under_self-crossing_change_2}, we have the following theorem immediately.

\begin{theorem}
\label{thm:main}
Let $X$ be a quasi-trivial quandle and $\theta \in \Hom{C^{Q, qt}_{2}(X), A}$ be a 2-cocycle with coefficients in an abelian group $A$.
Then, for each link $L$ and index $i$, the $i$-th quandle cocycle invariant $\Phi(L, \theta; i)$ is invariant under link-homotopy.
\end{theorem}

\begin{remark}
A \emph{trivial quandle} is a non-empty finite set $X$ equipped with a binary operation $\ast$ satisfying $x \ast y = x$ for each $x, y \in X$.
Carter et al.\ \cite{CJKLS2003} showed that, for each 2-cocycle $\theta \in \Hom{C^{Q}_{2}(X), A}$ with coefficients in an abelian group $A$ and index $i$, the $i$-th quandle cocycle invariant $\Phi(L, \theta; i)$ is completely determined by the linking numbers of $L$.
Since the linking numbers are invariant under link-homotopy, $\Phi(L, \theta; i)$ is also invariant under link-homotopy.
It is easy to see that each trivial quandle is quasi-trivial and its 2-cocycle satisfies not only the conditions (C1) and (C2) but the condition (C3).
\end{remark}

\section{Example}
\label{sec:example}

In this final section, as an application of Theorem \ref{thm:main}, we show a famous fact that the Borromean rings is not trivial up to link-homotopy.

Let $X$ be a set consisting of twelve elements $a_{i}$, $b_{i}$ and $c_{i}$ ($1 \leq i \leq 4$).
Define a binary operation $\ast$ on $X$ by Table \ref{tbl:binary_operation} whose $(i+1, j+1)$-entry denotes $x \ast y$ with $x$ being the $(i+1, 1)$-entry and $y$ the $(1, j+1)$-entry.
Then $X$ with $\ast$ is in fact a quasi-trivial quandle.
Furthermore, define an element $\theta \in \Hom{C^{Q, qt}_{2}(X), \mathbb{Z}_{2}}$ by Table \ref{tbl:2-cocycle} whose $(i+1, j+1)$-entry denotes $\theta(x, y)$ with $x$ being the $(i+1, 1)$-entry and $y$ the $(1, j+1)$-entry.
It is routine to check that $\theta$ is a 2-cocycle.
\begin{table}[htbp]
\begin{tabular}{|c||c|c|c|c|c|c|c|c|c|c|c|c|} \hline
$\ast$ & $a_{1}$ & $a_{2}$ & $a_{3}$ & $a_{4}$ & $b_{1}$ & $b_{2}$ & $b_{3}$ & $b_{4}$ & $c_{1}$ & $c_{2}$ & $c_{3}$ & $c_{4}$ \\ \hline \hline
$a_{1}$ & $a_{1}$ & $a_{1}$ & $a_{1}$ & $a_{1}$ & $a_{2}$ & $a_{2}$ & $a_{2}$ & $a_{2}$ & $a_{3}$ & $a_{3}$ & $a_{3}$ & $a_{3}$ \\ \hline
$a_{2}$ & $a_{2}$ & $a_{2}$ & $a_{2}$ & $a_{2}$ & $a_{1}$ & $a_{1}$ & $a_{1}$ & $a_{1}$ & $a_{4}$ & $a_{4}$ & $a_{4}$ & $a_{4}$ \\ \hline
$a_{3}$ & $a_{3}$ & $a_{3}$ & $a_{3}$ & $a_{3}$ & $a_{4}$ & $a_{4}$ & $a_{4}$ & $a_{4}$ & $a_{1}$ & $a_{1}$ & $a_{1}$ & $a_{1}$ \\ \hline
$a_{4}$ & $a_{4}$ & $a_{4}$ & $a_{4}$ & $a_{4}$ & $a_{3}$ & $a_{3}$ & $a_{3}$ & $a_{3}$ & $a_{2}$ & $a_{2}$ & $a_{2}$ & $a_{2}$ \\ \hline
$b_{1}$ & $b_{3}$ & $b_{3}$ & $b_{3}$ & $b_{3}$ & $b_{1}$ & $b_{1}$ & $b_{1}$ & $b_{1}$ & $b_{2}$ & $b_{2}$ & $b_{2}$ & $b_{2}$ \\ \hline
$b_{2}$ & $b_{4}$ & $b_{4}$ & $b_{4}$ & $b_{4}$ & $b_{2}$ & $b_{2}$ & $b_{2}$ & $b_{2}$ & $b_{1}$ & $b_{1}$ & $b_{1}$ & $b_{1}$ \\ \hline
$b_{3}$ & $b_{1}$ & $b_{1}$ & $b_{1}$ & $b_{1}$ & $b_{3}$ & $b_{3}$ & $b_{3}$ & $b_{3}$ & $b_{4}$ & $b_{4}$ & $b_{4}$ & $b_{4}$ \\ \hline
$b_{4}$ & $b_{2}$ & $b_{2}$ & $b_{2}$ & $b_{2}$ & $b_{4}$ & $b_{4}$ & $b_{4}$ & $b_{4}$ & $b_{3}$ & $b_{3}$ & $b_{3}$ & $b_{3}$ \\ \hline
$c_{1}$ & $c_{2}$ & $c_{2}$ & $c_{2}$ & $c_{2}$ & $c_{3}$ & $c_{3}$ & $c_{3}$ & $c_{3}$ & $c_{1}$ & $c_{1}$ & $c_{1}$ & $c_{1}$ \\ \hline
$c_{2}$ & $c_{1}$ & $c_{1}$ & $c_{1}$ & $c_{1}$ & $c_{4}$ & $c_{4}$ & $c_{4}$ & $c_{4}$ & $c_{2}$ & $c_{2}$ & $c_{2}$ & $c_{2}$ \\ \hline
$c_{3}$ & $c_{4}$ & $c_{4}$ & $c_{4}$ & $c_{4}$ & $c_{1}$ & $c_{1}$ & $c_{1}$ & $c_{1}$ & $c_{3}$ & $c_{3}$ & $c_{3}$ & $c_{3}$ \\ \hline
$c_{4}$ & $c_{3}$ & $c_{3}$ & $c_{3}$ & $c_{3}$ & $c_{2}$ & $c_{2}$ & $c_{2}$ & $c_{2}$ & $c_{4}$ & $c_{4}$ & $c_{4}$ & $c_{4}$ \\ \hline
\end{tabular}
\bigskip
\caption{}
\label{tbl:binary_operation}
\end{table}
\begin{table}[htbp]
\begin{tabular}{|c||c|c|c|c|c|c|c|c|c|c|c|c|} \hline
$\theta$ & $a_{1}$ & $a_{2}$ & $a_{3}$ & $a_{4}$ & $b_{1}$ & $b_{2}$ & $b_{3}$ & $b_{4}$ & $c_{1}$ & $c_{2}$ & $c_{3}$ & $c_{4}$ \\ \hline \hline
$a_{1}$ & $0$ & $0$ & $0$ & $0$ & $1$ & $0$ & $1$ & $0$ & $1$ & $1$ & $0$ & $0$ \\ \hline
$a_{2}$ & $0$ & $0$ & $0$ & $0$ & $0$ & $1$ & $0$ & $1$ & $0$ & $0$ & $1$ & $1$ \\ \hline
$a_{3}$ & $0$ & $0$ & $0$ & $0$ & $1$ & $0$ & $1$ & $0$ & $1$ & $1$ & $0$ & $0$ \\ \hline
$a_{4}$ & $0$ & $0$ & $0$ & $0$ & $0$ & $1$ & $0$ & $1$ & $0$ & $0$ & $1$ & $1$ \\ \hline
$b_{1}$ & $1$ & $1$ & $0$ & $0$ & $0$ & $0$ & $0$ & $0$ & $1$ & $0$ & $1$ & $0$ \\ \hline
$b_{2}$ & $0$ & $0$ & $1$ & $1$ & $0$ & $0$ & $0$ & $0$ & $0$ & $1$ & $0$ & $1$ \\ \hline
$b_{3}$ & $1$ & $1$ & $0$ & $0$ & $0$ & $0$ & $0$ & $0$ & $1$ & $0$ & $1$ & $0$ \\ \hline
$b_{4}$ & $0$ & $0$ & $1$ & $1$ & $0$ & $0$ & $0$ & $0$ & $0$ & $1$ & $0$ & $1$ \\ \hline
$c_{1}$ & $1$ & $0$ & $1$ & $0$ & $1$ & $1$ & $0$ & $0$ & $0$ & $0$ & $0$ & $0$ \\ \hline
$c_{2}$ & $0$ & $1$ & $0$ & $1$ & $0$ & $0$ & $1$ & $1$ & $0$ & $0$ & $0$ & $0$ \\ \hline
$c_{3}$ & $1$ & $0$ & $1$ & $0$ & $1$ & $1$ & $0$ & $0$ & $0$ & $0$ & $0$ & $0$ \\ \hline
$c_{4}$ & $0$ & $1$ & $0$ & $1$ & $0$ & $0$ & $1$ & $1$ & $0$ & $0$ & $0$ & $0$ \\ \hline
\end{tabular}
\bigskip
\caption{}
\label{tbl:2-cocycle}
\end{table}

Suppose $D_{T}$ and $D_{B}$ are diagrams of the trivial 3-component link $L_{T}$ and the Borromean rings $L_{B}$ depicted in Figure \ref{fig:example} respectively.
Since $D_{T}$ has no crossing, for each arc coloring $\mathscr{A}_{T}$ of $D_{T}$ by $X$ and index $i$, $W(\mathscr{A}_{T}, \theta; i)$ is always equal to 0.
On the other hand, we have an arc coloring $\mathscr{A}_{B}$ of $D_{B}$ by $X$, depicted in Figure \ref{fig:example}, satisfying $W(\mathscr{A}_{B}, \theta; i) = 1$ for each index $i$.
It says that $\Phi(L_{T}, \theta; i) \neq \Phi(L_{B}, \theta; i)$ for each index $i$.
Thus, by Theorem \ref{thm:main}, the Borromean rings $L_{B}$ is not link-homotopic to the trivial 3-component link $L_{T}$.
\begin{figure}
\begin{center}
\includegraphics[scale=0.25]{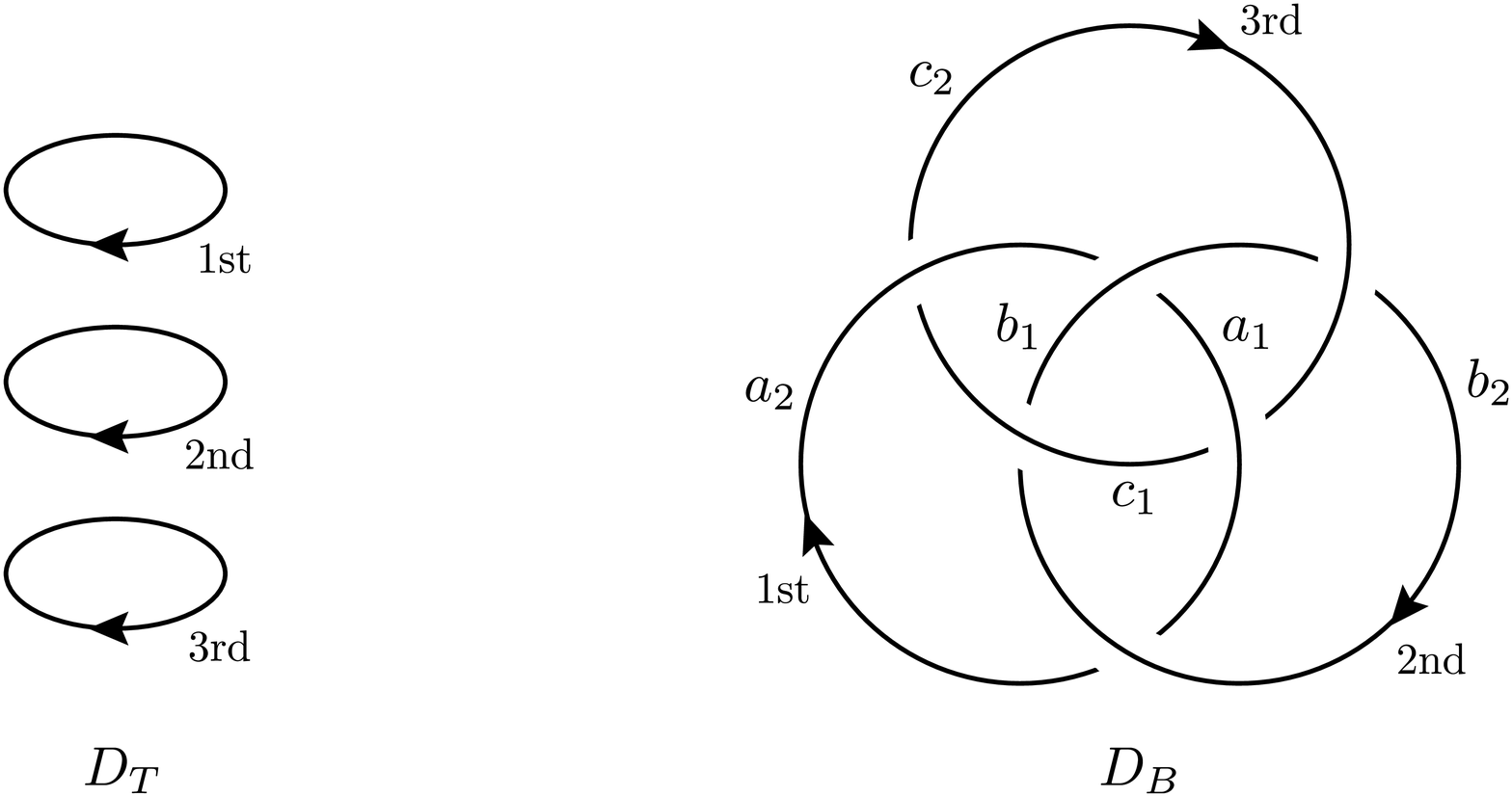}
\end{center}
\caption{}
\label{fig:example}
\end{figure}

\begin{remark}
Since the number of all arc colorings of $D_{B}$ by $X$ is equal to the number for $D_{T}$, we essentially have to compute the quandle cocycle invariants to distinguish $L_{B}$ and $L_{T}$ up to link-homotopy.
\end{remark}

\section*{Appendix. Decomposition of quandle cocycle invariant}
\label{app:decomposition_of_quandle_cocycle_invariant}

In this appendix, we see that each quandle cocycle invariant can be decomposed into multisets which are still invariant under ambient isotopy or link-homotopy.
It seems that these multisets are more useful than a quandle cocycle invariant itself to classify links.

We first note that the inner automorphism group of a quandle naturally acts on the quandle.
Therefore each quandle is decomposed into the orbits of the inner automorphism group.

Let $X$ be a quandle and $\cup_{\lambda \in \Lambda} X_{\lambda}$ the orbit decomposition of $X$.
Remark that each arc coloring of a link diagram maps arcs belonging to the same component in the same orbit.
We thus call an arc coloring of a diagram of an $n$-component link by $X$ to be in $(X_{\lambda_{1}}, X_{\lambda_{2}}, \dots, X_{\lambda_{n}})$ if it maps arcs belonging to the $i$-th component in $X_{\lambda_{i}}$.
Obviously, the number of all arc colorings in $(X_{\lambda_{1}}, X_{\lambda_{2}}, \dots, X_{\lambda_{n}})$ is invariant under ambient isotopy for each $n$-tuple $(\lambda_{1}, \lambda_{2}, \dots, \lambda_{n}) \in \Lambda^{n}$.

Let $\theta \in \Hom{C^{Q}_{2}(X), A}$ be a 2-cocycle with coefficients in an abelian group $A$.
For an $n$-component link $L$ and $n$-tuple $(\lambda_{1}, \lambda_{2}, \dots, \lambda_{n}) \in \Lambda^{n}$, consider the multiset $\Phi(L, \theta; X_{\lambda_{1}}, X_{\lambda_{2}}, \dots, X_{\lambda_{n}}; i)$ consisting of $i$-th weights of all arc colorings of a diagram of $L$ in $(X_{\lambda_{1}}, X_{\lambda_{2}}, \dots, X_{\lambda_{n}})$.
Then obviously
\[
\Phi(L, \theta; i) = \bigcup_{(\lambda_{1}, \lambda_{2}, \dots, \lambda_{n}) \in \Lambda^{n}} \Phi(L, \theta; X_{\lambda_{1}}, X_{\lambda_{2}}, \dots, X_{\lambda_{n}}; i)
\]
and each $\Phi(L, \theta; X_{\lambda_{1}}, X_{\lambda_{2}}, \dots, X_{\lambda_{n}}; i)$ is invariant under ambient isotopy.

Of course, the above arguments hold even if $X$ is a quasi-trivial quandle and $\theta \in \Hom{C^{Q, qt}_{2}(X), A}$ a 2-cocycle.
In this case, the number of all arc colorings in $(X_{\lambda_{1}}, X_{\lambda_{2}}, \dots, X_{\lambda_{n}})$ and the multiset $\Phi(L, \theta; X_{\lambda_{1}}, X_{\lambda_{2}}, \dots, X_{\lambda_{n}}; i)$ are invariant under link-homotopy.

\bibliographystyle{amsplain}

\end{document}